\documentclass[12pt,francais]{article}
\usepackage{amsmath}\usepackage{epsf,amsfonts,amsthm}\usepackage{amscd,amssymb}
\usepackage{xcolor,epic,eepic}\usepackage{epsfig}
\usepackage{indentfirst, delarray}
\usepackage{rotating}
\usepackage{mathdots}
\usepackage[matrix,arrow,curve]{xy}
\usepackage{graphicx}
\usepackage{fancyhdr}
\usepackage{dsfont,texdraw}
\usepackage{mathrsfs}
\usepackage{latexsym}
\usepackage{hyperref}
\usepackage{tikz-cd}
\usepackage{tikz, graphicx}
\theoremstyle{plain}

\usepackage{tikz}
\usetikzlibrary{arrows,chains,matrix,positioning,scopes,snakes}
\makeatletter
\tikzset{join/.code=\tikzset{after node path={%
\ifx\tikzchainprevious\pgfutil@empty\else(\tikzchainprevious)%
edge[every join]#1(\tikzchaincurrent)\fi}}}
\makeatother

\tikzset{>=stealth',every on chain/.append style={join},
         every join/.style={->}}
\tikzset{
    >=stealth',
    punkt/.style={
           rectangle,
           rounded corners,
           draw=black, very thick,
           text width=6.5em,
           minimum height=2em,
           text centered},
    pil/.style={
           ->,
           thick,
           shorten <=2pt,
           shorten >=2pt,}
}

\newcommand{\bee}{\begin{enumerate}}
\newcommand{\eee}{\end{enumerate}}
\newcommand{\benn}{\begin{equation*}}
\newcommand{\eenn}{\end{equation*}}
\newcommand{\be}{\begin{equation}}
\newcommand{\ee}{\end{equation}}
\newcommand{\bean}{\begin{eqnarray}}
\newcommand{\eean}{\end{eqnarray}}
\newcommand{\bea}{\begin{eqnarray*}}
\newcommand{\eea}{\end{eqnarray*}}

\newcommand{\Z}{\mathbb{Z}}
\newcommand{\R}{\mathbb{R}}

\newcommand{\op}[1]{\!\!\mathop{\rm ~#1}\nolimits}
\newcommand{\id}{\op{id}}

\newcommand{\zig}{\approx }
\mathchardef\za="710B  
\mathchardef\zb="710C  
\mathchardef\zg="710D  
\mathchardef\zd="710E  
\mathchardef\zve="710F 
\mathchardef\zz="7110  
\mathchardef\zh="7111  
\mathchardef\zy="7112 

\mathchardef\zi="7113  
\mathchardef\zk="7114  
\mathchardef\zl="7115  
\mathchardef\zm="7116  
\mathchardef\zn="7117  
\mathchardef\zx="7118  
\mathchardef\zp="7119  
\mathchardef\zr="711A  
\mathchardef\zs="711B  
\mathchardef\zt="711C  
\mathchardef\zu="711D  
\mathchardef\zf="711E 
\mathchardef\zq="711F  
\mathchardef\zc="7120  
\mathchardef\zw="7121  
\mathchardef\ze="7122  
\mathchardef\zvy="7123  
\mathchardef\zvw="7124  
\mathchardef\zvr="7125 
\mathchardef\zvs="7126 
\mathchardef\zvf="7127  
\mathchardef\zG="7000  
\mathchardef\zD="7001  
\mathchardef\zY="7002  
\mathchardef\zL="7003  
\mathchardef\zX="7004  
\mathchardef\zP="7005  
\mathchardef\zS="7006  
\mathchardef\zU="7007  
\mathchardef\zF="7008  
\mathchardef\zW="700A  

 \newcommand{\cD}{{\cal D}}

\newtheorem{rem}{Remark}
\newtheorem{theo}{Theorem}
\newtheorem{prop}{Proposition}
\newtheorem{lem}{Lemma}
\newtheorem{cor}{Corollary}

\newtheorem{defi}{Definition}

\definecolor{cof}{RGB}{219,144,71}
\definecolor{pur}{RGB}{186,146,162}
\definecolor{greeo}{RGB}{91,173,69}
\definecolor{greet}{RGB}{52,111,72}

\begin{document}

\title{Indeterminacies and models of homotopy limits}
\author{Alisa Govzmann, Damjan Pi\v{s}talo, and Norbert Poncin}
\maketitle

\begin{abstract} In \cite{CompTheo} we studied the indeterminacy of the value of a derived functor at an object using different definitions of a derived functor and different types of fibrant replacement. In the present work we focus on derived or homotopy limits, which of course depend on the model structure of the diagram category under consideration. The latter is not necessarily unique, which is an additional source of indeterminacy. In the case of homotopy pullbacks, we introduce the concept of full homotopy pullback by identifying the homotopy pullbacks associated with three different model structures of the category of cospan diagrams, thus increasing the number of canonical representatives. Finally, we define generalized representatives or models of homotopy limits and full homotopy pullbacks. The concept of model is a unifying approach that includes the homotopy pullback used in \cite{JL} and the homotopy fiber square defined in \cite{Hir} in right proper model categories. Properties of the latter are generalized to models in any model category. \end{abstract}

\small{\vspace{2mm} \noindent {\bf MSC 2020}: 18E35, 18N40, 14A30 \medskip

}

\thispagestyle{empty}

\section{Introduction}

The classical Kan homotopy category (resp., the Quillen homotopy category) of a model category $\tt M$ is a faint (resp., a strong) [and even a weak (resp., a strict)] localization of $ \tt M $ at its class of weak equivalences. This leads to total derived functors defined as Kan extension, to faintly universal total derived functors and strongly universal ones, for which we proved comparison theorems in \cite{CompTheo}. In particular, Kan and strong right derived functors of a functor which preserves weak equivalences between fibrant objects exist, coincide up to a natural isomorphism, and are given by different types of fibrant replacement (see Theorem \ref{Fundamental0} below). The value of a derived functor at an object belongs therefore to a well-defined isomorphism class of the target {\it homotopy} category, regardless of the definition of a derived functor and the type of fibrant replacement we use. Considered as an object of the target {\it model} category, such a value is thus only well-defined up to a zigzag of weak equivalences (see Theorem \ref {FundamentalCor}). The core topic of the present work is an efficient handling of this indeterminacy.\medskip

Since a derived functor depends on the model structures of the source and target of the original functor, the derived limit functor or homotopy limit functor depends on the model structures of the small diagram category $\tt D:= Fun(S,M)$ and of the underlying category $\tt M$ considered. The model structure $\zs$ of $\tt D$ is not necessarily unique, even if the model structure of the ambient category $\tt M$ remains unchanged and if we take into account that the limit functor must be a right Quillen functor with respect to $\zs\,$. The resulting freedom in choosing $\zs$ is an additional source of indeterminacy (see Theorem \ref{IndeterminacyHoLimTheo}). \medskip

In the case ${\tt S}=\{c\to d\leftarrow b\}$ the diagrams $\tt Fun(S,M)$ are the cospan diagrams $C\to D\leftarrow B$ of $\tt M$ and the homotopy limit is the homotopy pullback. The category of cospan diagrams can be equipped with three Reedy model structures $\zs_i$ ($i\in\{1,2,3\}$) with respect to which the pullback functor is a right Quillen functor. The homotopy pullbacks of a cospan $C\to D\leftarrow B$ with respect to the $\zs_i$ admit as canonical representatives the standard pullbacks of the corresponding $\zs_i$-fibrant replacements of $C\to D\leftarrow B\,.$ We define the full homotopy pullback of $C\to D\leftarrow B$ by identifying its homotopy pullbacks with respect to the $\zs_i\,,$ thus increasing the number of canonical representatives (see Theorem \ref{IndHoPull}).\medskip

We further enhance the flexibility of homotopy limits by defining generalized representatives, also referred to as models. The concept of model is valid in every model category but the model condition can be relaxed in right proper model categories (see Theorem \ref{IndHoPullMod} and Theorem \ref{RightProper}). Models are a unifying approach that captures the notion of homotopy pullback that is used in \cite{JL} and the notion of homotopy fiber square that is defined in right proper model categories equipped with a fixed functorial factorization system in \cite{Hir} (see Corollary \ref{CorLur} and Corollary \ref{HoFibSqCor}). Most results of homotopy fiber squares in right proper model categories remain valid for models or model squares in all model categories (see for example Proposition \ref{PasLaw} and Proposition \ref{ComPara}).\medskip

Most of the results in this paper are not new, but a structured rigorous presentation in an appropriate unifying context does not seem to exist. The proven theorems offer orientation in an environment with many indeterminacies and show that the standard concepts concerned have a pretty good stability with regard to all the necessary choices. Applications can be expected in homotopic algebraic geometry \cite{TV05, TV08, Schwarz, Linear, HAC} and higher supergeometry \cite{Berezinian, Z2nManifolds, LocalForms,Integration}. Indeed, these areas make extensive use of the functor of points and are the contexts from which the need arose to examine the subjects of this paper.\medskip

{\it Conventions and notations}. We assume that the reader is familiar with model categories. Although we use many results of \cite{CompTheo}, we paid attention to independent readability when writing this work. We adopt the definition of a model category that is used in \cite{Hir}. More precisely, a model category is a category $\tt M$ that is equipped with three classes of morphisms called weak equivalences, fibrations and cofibrations. The category $\tt M$ has all small limits and colimits and the 2-out-of-3 axiom, the retract axiom and the lifting axiom are satisfied. Moreover $\tt M$ admits a functorial cofibration - trivial fibration factorization system (Cof - TrivFib factorization) and a functorial trivial cofibration - fibration factorization system (TrivCof - Fib factorization). Furthermore, we work with the Quillen homotopy category $\tt Ho(M)$ of $\tt M\,$, which is a strict localization ${\tt M}[[W^{-1}]]$ of $\tt M$ at its weak equivalences $W$ with localization functor denoted $\zg_{\tt M}\,,$ and we use the Kan extension derived functor operations $\mathbb{L}^{\op{K}},\mathbb{R}^{\op{K}}$ and the strongly universal derived functor operations $\mathbb{L}^{\op{S}},\mathbb{R}^{\op{S}}$ in the sense of \cite{CompTheo}.

\section{Indeterminacy of a derived functor}

For the definition of the left versions of the preceding derived functor operations, existence and uniqueness results, and the type of `commutation' relation they satisfy, we refer the reader to Definition 8 and Propositions 5, 11, 12, 13 in \cite{CompTheo}. The right versions are dual.

\begin{rem}
\emph{Let us emphasize very clearly that the {\small K} derived functors $\mathbb{L}^{\op{K}}F$ and $\mathbb{R}^{\op{K}}G$ of functors $F$ and $G$ between model categories were defined in \cite{CompTheo} up to a unique natural isomorphism and that the symbols $\mathbb{L}^{\op{K}}F$ and $\mathbb{R}^{\op{K}}G$ stand for an arbitrary representative. While in \cite{CompTheo} we systematically identified objects $A$ and $B$ that are connected by a unique or canonical isomorphism, in this text we usually write $A\doteq B$ and not $A=B\,.$ Similarly the {\small S} derived functors $\mathbb{L}^{\op{S}}F$ and $\mathbb{R}^{\op{S}}G$ were defined up to a (not necessarily unique) natural isomorphism and the symbols stand again for any representative.}
\end{rem}

Let $(\za,\zb)$ be any functorial TrivCof - Fib factorization system. For every object $X\in\tt M\,,$ it factors the map $t_X:X\to \ast$ to the terminal object of $\tt M$ into a trivial cofibration $r_X:=\za(t_X)$ followed by a fibration $\zb(t_X)\,$: $$t_X:X\stackrel{\sim}{\rightarrowtail} RX \twoheadrightarrow \ast\;.$$ Regardless of the factorization $$t_X:X\stackrel{\sim}{\to}FX\twoheadrightarrow \ast$$ of $t_X:X\to\ast$ into a weak equivalence $f_X$ followed by a fibration considered, we refer to $FX$ as {\it a} {\it fibrant replacement} of $X\,.$ The object $RX$ we call a {\it fibrant C-replacement} of $X$ (or just a fibrant replacement if we do not want to stress that $r_X$ is a cofibration). From the fact that the factorization $(\za,\zb)$ is functorial it follows that $R$ is an endofunctor of $\tt M\,$. Moreover $r_X:X\to RX$ is functorial in $X\,$: $r$ is a natural transformation $r:\op{id}_{\tt M}\Rightarrow R$ from the identity functor $\op{id}_{\tt M}$ to the {\it fibrant replacement functor} $R$ \cite{Ho99}. Instead of the fibrant C-replacement functor $R$ that is globally defined by the functorial factorization $(\za,\zb)\,,$ we will also use local / object-wise fibrant replacements $FX$ or {\it local fibrant C-replacements} $\tilde{F}X$ such that the map $f_X$ in the factorization $$t_X:X\stackrel{\sim}{\rightarrowtail}\tilde{F}X{\twoheadrightarrow}\ast$$ is $\op{id}_X$ if $X$ is already fibrant. If for every $X$ we choose such a local fibrant C-replacement and if $f:X\to Y\,,$ there is a lifting $\tilde{F}f:\tilde{F}X\to \tilde{F}Y,$ which will play an important role:

\begin{equation}\label{FTilde} \begin{tikzpicture}
 \matrix (m) [matrix of math nodes, row sep=3em, column sep=3em]
   {  \stackrel{}{X}  & \stackrel{}{Y} & \tilde{F}Y  \\
      \tilde{F}X & & \stackrel{}{\ast}  \\ };
 \path[->]
 (m-1-1) edge node[auto] {\small{$f$}} (m-1-2)
 (m-1-2) edge [>->] node[auto] {\small{$\;\;{}_{\widetilde{}}\,\;f_Y$}}(m-1-3)
 (m-2-1) edge [->>] (m-2-3)
 (m-1-1) edge [>->] node[auto] {\small{$\;\;{}_{\widetilde{}}\,\;f_X$}} (m-2-1)
 (m-1-3) edge [->>] (m-2-3)
 (m-2-1) edge [->, dashed] node[auto] {\small{$\tilde{F}f$}} (m-1-3);
\end{tikzpicture}
\end{equation}

From \cite{CompTheo} we also know:

\begin{theo}\label{Fundamental0}
If $G\in\tt Fun(M,N)$ is a functor between model categories that sends weak equivalences between fibrant objects to weak equivalences, its Kan extension right derived functor $$\mathbb{R}^{\op{K}}G\in\tt Fun(Ho(M),Ho(N))$$ and its strongly universal right derived functor $$\mathbb{R}^{\op{S}}G\in\tt Fun(Ho(M),Ho(N))$$ exist and we have \be\label{Fundamental1}\mathbb{R}^{\op{K}}G\doteq\op{Ho}(\zg_{\tt N}\circ G\circ\tilde{F})\doteq\mathbb{R}_R^{\op{S}}G:=\op{Ho}(\zg_{\tt N}\circ G\circ R)\stackrel{\cong}{\Rightarrow} \mathbb{R}^{\op{S}}G\;,\ee where $\tilde{F}$ is a local fibrant C-replacement, $R$ is a fibrant C-replacement functor and $\op{Ho}$ the unique on the nose factorization through $\tt Ho(M)\,$.  This implies that \be\label{Fundamental2}\mathbb{R}^{\op{K}}G\circ\zg_{\tt M}\doteq\zg_{\tt N}\circ G\circ\tilde{F}\doteq\mathbb{R}_R^{\op{S}}G\circ\zg_{\tt M}=\zg_{\tt N}\circ G\circ R\stackrel{\cong}{\Rightarrow}\mathbb{R}^{\op{S}}G\circ\zg_{\tt M}\;,\ee where $\doteq$ denotes a canonical natural isomorphism and $\stackrel{\cong}{\Rightarrow}$ a not necessarily canonical natural isomorphism.
\end{theo}

Hence, for every $X\in\tt M\,,$ the value of the derived functor at $\zg_{\tt M}X=X\in\tt Ho(M)$ is \be\label{Indeterminacy}\mathbb{R}^{\op{K}}G(X)\doteq G(\tilde{F}X)\doteq\mathbb{R}_R^{\op{S}}G(X)=G(R X)\cong\mathbb{R}^{\op{S}}G(X)\;,\ee where $\doteq$ denotes a canonical isomorphism in $\tt Ho(N)$ and $\cong$ a not necessarily canonical isomorphism in $\tt Ho(N)\,$.

\begin{rem}
\emph{Since $\mathbb{R}^{\op{K}}G$ (resp., $\mathbb{R}^{\op{S}}G$) is defined up to a canonical natural isomorphism (resp., up to a natural isomorphism) \cite{CompTheo}, the results of \eqref{Fundamental1} are the best possible ones.}
\end{rem}

The next diagram shows that if $FX$ is any fibrant replacement of $X\,,$ there is a lifting $\ell:\tilde{F}X\to FX\,:$

\begin{equation}\label{FibRep}
\begin{tikzpicture}
 \matrix (m) [matrix of math nodes, row sep=3em, column sep=3em]
   {X & & FX   \\
    \tilde{F}X & & \stackrel{}{\ast} \\};
 \path[->]
 (m-1-1) edge [->] node[above]{\small{$\sim$}} node[below]{\small{$f_X$}} (m-1-3)
 (m-1-1) edge [>->] node[left]{\small{$\sim$}} node[right]{\small{$\tilde{f}_X$}} (m-2-1)
 (m-2-1) edge [->>] (m-2-3)
 (m-1-3) edge [->>](m-2-3)
 (m-2-1) edge [->, dashed] node[below] {\small{$\ell$}} (m-1-3);
\end{tikzpicture}
\end{equation}
Since $\ell$ is a weak $\tt M$-equivalence between fibrant objects, its image $G(\ell)$ is a weak $\tt N$-equivalence \be\label{WeqMDF}G(\tilde{F}X)\stackrel{\sim}{\to}G(FX)\ee and the image $\zg_{\tt N}(G(\ell))$ is a $\tt Ho(N)$-isomorphism \be\label{PreIndeterminacy2}G(\tilde{F}X)\cong G(FX)\;.\ee

\begin{prop} The isomorphism \eqref{PreIndeterminacy2} is canonical: \be\label{Indeterminacy2}G(\tilde{F}X)\doteq G(FX)\;.\ee\end{prop}

\begin{proof} Take two different $\tt M$-morphisms $\ell_i:\tilde{F}X\to FX$ ($i\in\{1,2\}$) that render the upper triangle in Diagram \ref{FibRep} commutative, so that $\ell_1\circ\tilde{f}_X=\ell_2\circ\tilde{f}_X\,.$ Since $Y:=FX$ is fibrant and $\tilde{f}_X:X\stackrel{\sim}{\rightarrowtail}\tilde{F}X$ is a trivial cofibration, right composition by $\tilde{f}_X$ induces a 1:1 correspondence between right homotopy classes of morphisms (the result is well known; we gave the proof of its dual in \cite{CompTheo}), we get $\ell_1\simeq^r\ell_2\,.$ This means that $\ell_1\times\ell_2:\tilde{F}X\to Y\times Y$ factors through a path object $\op{Path}(Y)$ of $Y\,,$ i.e., that there is a factorization \be\label{CanIso1}p_1\circ w:=\psi_1\circ p\circ w=\id_Y\quad\text{and}\quad p_2\circ w:=\psi_2\circ p\circ w=\id_Y\;,\ee where $\psi_1,\psi_2:Y\times Y\to Y\,,$ $w:Y\stackrel{\sim}{\to}\op{Path}(Y)$ and $p:\op{Path}(Y)\twoheadrightarrow Y\times Y,$ and a factorization \be\label{CanIso2}p_1\circ K=\ell_1\quad\text{and}\quad p_2\circ K=\ell_2\;,\ee where $K:\tilde{F}X\to \op{Path}(Y)\,.$ From \eqref{CanIso1} it follows that $p_i:\op{Path}(Y)\to Y$ is a weak equivalence between fibrant objects (indeed, since fibrations are closed under pullbacks and compositions, the product of fibrant objects and the path object of a fibrant object are fibrant). If we apply $\zg_{\tt N}\circ G$ to \eqref{CanIso1} and remember that $\zg_{\tt N}(G(p_i))$ is an isomorphism in view of the assumption on $G\,,$ we see that $\zg_{\tt N}(G(w))$ is the inverse isomorphism and that $\zg_{\tt N}(G(p_1))=\zg_{\tt N}(G(p_2))\,.$ It now follows from \eqref{CanIso2} that \be\label{IndLift}\zg_{\tt N}(G(\ell_1))=\zg_{\tt N}(G(\ell_2))\;,\ee so that the $\tt Ho(N)$-isomorphism $\zg_{\tt N}(G(\ell))$ is canonically implemented by the replacements $\tilde{F}X$ and $FX\,.$ \end{proof}

Notice now that if $X,Y\in\tt M$ are related by a zigzag of weak $\tt M$-equivalences, $X\stackrel{\sim}{\to} U\stackrel{\sim}{\leftarrow} V\stackrel{\sim}{\to} Y$ say, it suffices to apply the localization functor $\zg_{\tt M}$ to see that $X$ and $Y$ are isomorphic in $\tt Ho(M)\,.$ The converse is also true:

\begin{prop}
Two objects of a model category $\tt M$ are isomorphic as objects of $\tt Ho(M)$ if and only if they are related by a zigzag of weak equivalences of $\,\tt M\,.$
\end{prop}

\begin{proof}
Let $X,Y\in\tt M$ and let $$[\zvf]\in \op{Hom}_{\tt Ho(M)}(X,Y)=\op{Hom}_{\tt M}(RQX,RQY)/\simeq$$ be an isomorphism with inverse $[\psi]\,,$ where $R$ is a fibrant C-replacement functor and $Q$ a cofibrant F-replacement functor. Since $[\psi\circ\zvf]=[\id_{RQX}]$ and $[\zvf\circ\psi]=[\id_{RQY}]\,,$ we conclude that $\zvf$ and $\psi$ are inverse homotopy equivalences between fibrant-cofibrant objects, i.e., that they are weak equivalences. Therefore we have a zigzag of weak equivalences $$X\stackrel{\sim}{\twoheadleftarrow}QX\stackrel{\sim}{\rightarrowtail}RQX\stackrel{\zvf}{\to}RQY\stackrel{\sim}{\leftarrowtail}QY\stackrel{\sim}{\twoheadrightarrow}Y\;.$$
\end{proof}

The previous observations clarify the {\it indeterminacy of a value of a derived functor}:

\begin{rem}
\emph{In view of \eqref{Indeterminacy} and \eqref{Indeterminacy2} the value of a derived functor at an object is well defined only up to isomorphism of the target homotopy category. The isomorphism class is independent of the type of derived functor considered, Kan extension or strongly universal, as well as independent of the type of fibrant C-replacement considered, local or global. Also the choice of another local or another global replacement does not change the isomorphism class. If we compute the value of the derived functor using a local fibrant replacement that is not necessarily a C-replacement, we get again the same class. Finally, the three representatives considered of the value of the derived functor are related by canonical isomorphisms when viewed as objects of $\tt Ho(N)$ and by zigzags of weak equivalences when viewed as objects of $\tt N\,$.}
\end{rem}

\begin{rem}\label{RemNot} \emph{In the following we write $X\approx Y$ if $X$ and $Y$ are related by a zigzag of weak equivalences and we write $X\stackrel{\sim}{\rightleftarrows}Y$ if there is a weak equivalence from $X$ to $Y$ {\it and} a weak equivalence from $Y$ to $X$.} \end{rem}

If we use the notation of Remark \ref{RemNot}, Theorem \ref{Fundamental0} and Equation \eqref{Indeterminacy} imply

\begin{theo}\label{FundamentalCor}
Let $G\in\tt Fun(M,N)$ be a functor between model categories that sends weak $\tt M$-equivalences between fibrant objects to weak $\tt N$-equivalences. For every $X\in\tt M\,,$ the value at $X$ of the right derived functors $$\R^{\op{K}} G,\,\R^{\op{S}}G\in\tt Fun(Ho(M),Ho(N))$$ of $G$ is defined as an object of $\tt N$ up to a zigzag of weak $\tt N$-equivalences and its representatives are computed using a fibrant C-replacement functor $R\,,$ a local fibrant C-replacement $\tilde{F}$ or any local fibrant replacement $F$:
\be\label{Indeterminacy3}\R^{\op{K}}G(X)\approx \R^{\op{S}}G(X)\approx G(RX)\stackrel{\sim}{\rightleftarrows}G(\tilde{F}X)
\stackrel{\sim}{\rightarrow}G(FX)\;.\ee
\end{theo}

The last weak equivalence in Equation \eqref{Indeterminacy3} is the one of Equation \eqref{WeqMDF}. We get the weak equivalences between $G(RX)$ and $G(\tilde{F}X)$ just as we got the one from $G(\tilde{F}X)$ to $G(FX)\,.$\medskip

The dual versions of the results in this section for left derived functors are also true.

\section{Indeterminacy of a homotopy limit}

If $\tt S$ is a small category and $\tt M$ a model category, the functor category $\tt Fun(S,M)$ admits under mild conditions on the target category $\tt M$ an injective (resp., projective) model structure. The injective weak equivalences and cofibrations are defined as object-wise weak $\tt M$-equivalences and $\tt M$-cofibrations (the injective fibrations are characterized by their right lifting property with respect to the trivial cofibrations). These three classes satisfy the model category axioms, if $\tt M$ is a combinatorial model category. However, this condition is not a necessary one. If the axioms are satisfied, we refer to the resulting model structure as the {\bf injective model structure}. The {\bf projective model} structure is defined dually via object-wise weak equivalences and fibrations. A sufficient condition of existence is that the target category $\tt M$ is cofibrantly generated.\medskip

The constant functor $-^\ast:\tt M\to Fun(S,M)$ is the left adjoint to the limit functor: \be\label{AdjLim}-^\ast:\tt M\rightleftarrows Fun(S,M):\op{Lim}\ee If the injective model structure of $\tt Fun(S,M)$ exists, the constant functor respects cofibrations and trivial cofibrations and the adjunction is therefore a Quillen adjunction. More generally, let $\zs$ be any model structure on $\tt Fun(S,M)$ such that \eqref{AdjLim} is a Quillen adjunction $-^\ast \dashv\op{Lim}\,.$ It follows from Brown's lemma that a right (resp., left) Quillen functor sends weak equivalences between fibrant (resp., cofibrant) objects to weak equivalences, so that its right (resp., left) derived functor exists (see Theorem \ref{Fundamental0}). The left and right derived functors induced in homotopy by adjoint Quillen functors are themselves adjoint functors. The Quillen adjunction \eqref{AdjLim} induces therefore the adjunction $$\mathbb{L}_\zs(-^\ast):{\tt Ho(M)}\rightleftarrows {\tt Ho}({\tt Fun}_{\zs}{\tt (S,M)}):\mathbb{R}_{\zs}\!\op{Lim}\;.$$
This holds in both the case of K and S derived functors.

\begin{defi}\label{HoLimFunDef}
The derived functor $\mathbb{R}_{\zs}\!\op{Lim}$ is referred to as the {\bf homotopy limit functor} (more precisely, the K or the S homotopy limit functor) with respect to the model structure $\zs$ on diagrams.
\end{defi}
From Theorem \ref{FundamentalCor} follows the

\begin{theo}\label{IndeterminacyHoLimTheo}
Let $\,\tt S$ be a small category, let $\tt M$ be a model category and let $\zs$ be a model structure on the category $\tt Fun(S,M)$ of $\tt S$-shaped diagrams in $\tt M$ such that the adjunction \eqref{AdjLim} is a Quillen adjunction $-^\ast\dashv\op{Lim}\,.$ If $X\in\tt Fun(S,M)$ its homotopy limit with respect to $\zs$ is given as an object of $\tt M$ by \be\label{IndeterminacyHoLim}\mathbb{R}_{\zs}\!\op{Lim}(X)\zig\op{Lim}(R_{\zs}X)\stackrel{\sim}{\rightleftarrows}\op{Lim}(\tilde{F}_{\zs}X)
\stackrel{\sim}{\to}\op{Lim}(F_{\zs}X)\;,\ee where $\mathbb{R}_{\zs}\!\op{Lim}(X)$ can be interpreted as K or S derived functor and where $R_{\zs},\tilde{F}_{\zs},F_{\zs}$ are a fibrant C-replacement functor, a local fibrant C-replacement and a local fibrant replacement in the model structure $\zs$ on $\tt Fun(S,M)\,.$ The weak equivalence $\stackrel{\sim}{\to}$ is the universal morphism \be\label{UniMorHoLim}\op{Lim}(\ell_\zs):\op{Lim}(\tilde{F}_\zs X)\stackrel{\sim}{\rightarrow}\op{Lim}(F_\zs X)\ee that is induced by a lifting $\ell_\zs: \tilde{F}_\zs X\Rightarrow F_\zs X$ and its image $\zg_{\tt M}(\op{Lim}(\ell_\zs))$ in homotopy is independent of the lifting considered (see \eqref{IndLift}). A similar remark holds for the weak equivalences $\stackrel{\sim}{\rightleftarrows}\,$.
\end{theo}

In particular, if the injective model structure exists, we have:
\be\label{IndeterminacyHoLimInj}\mathbb{R}_{\op{inj}}\!\op{Lim}(X)\zig\op{Lim}(R_{\op{inj}}X)\stackrel{\sim}{\rightleftarrows}\op{Lim}(\tilde{F}_{\op{inj}}X)
\stackrel{\sim}{\to}\op{Lim}(F_{\op{inj}}X)\;.\ee

If the projective model structure exists, the dual result holds: $$\label{IndeterminacyHoColim}\mathbb{L}_{\op{proj}}\!\op{Colim}(X)\zig\op{Colim}(Q_{\op{proj}}X)\stackrel{\sim}{\rightleftarrows}\op{Colim}(\tilde{C}_{\op{proj}}X)\stackrel{\sim}{\leftarrow}
\op{Colim}(C_{\op{proj}}X)\;,$$ where $Q_{\op{proj}},\tilde{C}_{\op{proj}}$ and $C_{\op{proj}}$ are cofibrant replacements.

\begin{rem}
\emph{Equation \eqref{IndeterminacyHoLimInj} clarifies the indeterminacy of a small homotopy limit $\mathbb{R}_{\op{inj}}\!\op{Lim}(X)$ viewed as an object of the underlying model category in relation to the chosen definition of derived functors and the chosen replacement of the diagram under consideration. However, if the index category $\tt S$ is an appropriate Reedy category $\tt R\,$, the limit functor is also a right Quillen functor if the diagram category is equipped with its Reedy model structure. This leads to a homotopy limit $\mathbb{R}_{\op{Ree}}\!\op{Lim}(X)$ with respect to the Reedy structure and thus to another possible indeterminacy.}
\end{rem}

In the remainder of this section, we recall the results on Reedy categories and Reedy model structures that we need in the next section to explore the indeterminacy just mentioned. The exact understanding of all the indeterminacies is a prerequisite for the new approach to model categorical homotopy fiber sequences that we detail in \cite{NAHFS}.\medskip

If $\tt R$ is a Reedy category and $\tt M$ {\it any} model category, the functor category ${\tt Fun}({\tt R},{\tt M})$ can be equipped with a {\bf Reedy model structure}.\medskip

Reedy categories are defined using direct and inverse categories which are particularly simple examples of Reedy categories \cite{DHK}. A {\bf direct category} (resp., an {\bf inverse category}) is a small category that comes with a map $\op{deg}$ from objects to ordinals such that every non-identity morphism $r\to s$ raises (resp., lowers) the degree: $\op{deg} r<\op{deg} s$ (resp., $\op{deg} r>\op{deg} s$). A {\bf Reedy category} is a small category $\tt R$ together with two subcategories ${\tt R}_+$ and ${\tt R}_-$ which contain all the objects, and a map $\op{deg}$ from objects to ordinals such that:
\begin{enumerate}
  \item every $\tt R$-morphism factors uniquely into an ${\tt R}_-$-morphism and an ${\tt R}_+$-morphism,
  \item every non-identity ${\tt R}_-$-morphism lowers the degree,
  \item every non-identity ${\tt R}_+$-morphism raises the degree.
\end{enumerate}
For instance, the category ${\tt I}:=\{c\to d\leftarrow b\}$ is a direct category when equipped with degree map $\op{deg}_1$ defined by $\{0\to 2\leftarrow 1\}$, it is an inverse category for the degree map $\op{deg}_2$ defined by $\{1\to 0\leftarrow 2\}$ and it is a non-trivial Reedy category for $\op{deg}_3$ and $\op{deg}_4$ given by $\{0\to 1\leftarrow 2\}$ and $\{2\to 1\leftarrow 0\}\,.$\medskip

For every $X\in\tt Fun(R,M)$ and every $r\in\tt R$ one defines
\begin{enumerate}
  \item the {\bf latching object} $L_rX$ of $X$ at $r$ as the colimit $$L_rX:=\op{Colim}({\tt R}_{+r}\stackrel{\op{For}}{\to} {\tt R}\stackrel{X}{\to} {\tt M})=\op{Colim}(\op{For}^*X)\in{\tt M}\;,$$ where ${\tt R}_{+r}$ is the full subcategory of the over-category ${\tt R}_+\downarrow r$ (of all ${\tt R}_+$-morphisms with codomain $r$) that contains all the objects except the identity of $r\,,$ where $\op{For}$ is the forgetful functor, and where $\op{For}^*$ is the pullback functor,
  \item the {\bf matching object} $M_rX$ of $X$ at $r$ as the limit $$M_rX:=\op{Lim}({\tt R}^r_{-}\stackrel{\op{Rof}}{\to} {\tt R}\stackrel{X}{\to} {\tt M})=\op{Lim}(\op{Rof}^*X)\in{\tt M}\;,$$ where ${\tt R}^r_{-}$ is the full subcategory of the under-category $r\downarrow{\tt R}_-$ (of all ${\tt R}_-$-morphisms with domain $r$) that contains all the objects except the identity of $r\,,$ where $\op{Rof}$ is the forgetful functor, and where $\op{Rof}^*$ is the pullback functor.
  \end{enumerate}
For instance, let $r'\to r$ and $r''\to r$ be two non-identity morphisms of ${\tt R}_+$ and let $r'\to r''$ be a morphism of ${\tt R}_+$ that makes the resulting triangle commutative:

\begin{equation}
\begin{tikzpicture}
\matrix (m) [matrix of math nodes, row sep=2em, column sep=1.6em]
    {r' &  & r''     \\
			 & r &  \\
};
\path[->] (m-1-1)  edge[pil] (m-2-2) ;
\path[->] (m-1-3)  edge[pil] (m-2-2) ;
\path[->] (m-1-1)  edge[pil] (m-1-3) ;
\end{tikzpicture}
\end{equation}
\noindent The functor $\op{For}$ sends this morphism of ${\tt R}_{+r}$ to $r'\to r''$ and the functor $X$ sends the latter to $X_{r'}\to X_{r''}\,:$

\begin{equation}
\begin{tikzpicture}
\matrix (m) [matrix of math nodes, row sep=2em, column sep=1.6em]
    {X_{r'} &  & X_{r''}     \\
			 & L_rX &  \\
			& X_r  &   \\
			};
\path[->] (m-1-1) edge[pil,bend right=30] (m-3-2); 
\path[->] (m-1-3) edge[pil,bend left=30]  (m-3-2);
\path[->] (m-1-1)  edge[pil] (m-2-2) ;
\path[->] (m-1-3)  edge[pil] (m-2-2) ;
\path[dashed,->] (m-2-2)  edge[pil] (m-3-2) ; 
\path[->] (m-1-1)  edge[pil] (m-1-3) ;
\end{tikzpicture}
\end{equation}
\noindent The case of the matching object is dual, so that the universal {\bf latching morphism} $\ell_r(X):L_rX\to X_r$ in the previous diagram gets replaced by the {\bf matching morphism} $m_r(X):X_r\to M_rX\,.$ For every $r\in{\tt R}\,,$ we actually have a {\bf latching functor} and a {\bf latching natural transformation}  \be\label{LFLNT}L_r=\op{Colim}\circ\op{For}^*\in{\tt Fun(Fun(R,M),M)}\quad\text{and}\quad\ell_r:L_r\Rightarrow -_r\;,\ee where $-_r:{\tt Fun(R,M)}\to{\tt M}$ is the evaluation functor. Dually, we get a {\bf matching functor} and a {\bf matching natural transformation} \be\label{MFMNT}M_r=\op{Lim}\circ\op{Rof}^*\in{\tt Fun(Fun(R,M),M)}\quad\text{and}\quad m_r:-_r\Rightarrow M_r\;.\ee Hence, if $X,Y\in\tt Fun(R,M)$ and $f:X\Rightarrow Y\,,$ the following squares commute:

\begin{equation}
\begin{tikzpicture}
 \matrix (m) [matrix of math nodes, row sep=3em, column sep=3em]
   {L_rX & & X_r & & M_rX \\
    L_rY & & Y_r & & M_rY \\};
 \path[->]
 (m-1-1) edge [->] node[auto] {\small{$\;\;\ell_r(X)$}} (m-1-3)
 (m-1-3) edge [->] node[auto] {\small{$\;\;m_r(X)$}} (m-1-5)
 (m-2-1) edge [->] node[below] {\small{$\;\;\ell_r(Y)$}} (m-2-3)
 (m-2-3) edge [->] node[below] {\small{$\;\;m_r(Y)$}} (m-2-5)
 (m-1-1) edge [->] node[left] {\small{$L_rf$}} (m-2-1)
 (m-1-3) edge [->] node[left] {\small{$f_r$}}(m-2-3)
 (m-1-5) edge [->] node[right] {\small{$M_rf$}}(m-2-5);
\end{tikzpicture}
\end{equation}

The distinguished morphism classes of the Reedy model structure of $\tt Fun(R,M)$ that we mentioned above are defined as follows \cite{DHK}. The {\bf Reedy weak equivalences} are defined object-wise and are therefore the same as for the projective and injective model structures. The {\bf Reedy cofibrations} are the natural transformations $f:X\Rightarrow Y$ such that the induced universal $\tt M$-morphism \be\label{ReeCof}X_r\amalg_{L_rX}L_rY\to Y_r\ee is a cofibration for every $r\in\tt R\,.$ The {\bf Reedy fibrations} are the natural transformations $f:X\Rightarrow Y$ such that the induced universal $\tt M$-morphism \be\label{ReeFib}X_r\to Y_r\times_{M_rY}M_r X\ee is a fibration for every $r\in\tt R\,.$\medskip

If the Reedy category $\tt R$ is a direct category, its subcategory ${\tt R}_+$ is the full category ${\tt R}$ and its subcategory ${\tt R}_-$ is the discrete category that contains all the objects $r\in\tt R\,$. In this case the full subcategory ${\tt R}^r_{-}$ is the empty category, the functor $\op{Rof}^*X$ is the empty diagram and $M_rX$ is the terminal object $\ast$ of $\tt M$ for all $X\,.$ It follows from \eqref{ReeFib} that the Reedy fibrations are exactly the object-wise fibrations. Therefore the Reedy model structure of $\tt Fun(R,M)$ is the projective model structure. The dual result is also true:

\begin{rem}\label{RDI} \emph{If the Reedy category $\tt R$ is a direct (resp., an inverse) category, the Reedy model structure of $\,\tt Fun(R,M)$ is the projective (resp., the injective) model structure.}\end{rem}

If $\tt R$ is any Reedy category and {\bf $\tt M$ is a combinatorial model category}, hence a cofibrantly generated one, so that all three model structures exist, the Reedy model structure lies between the projective and the injective structures. More precisely, the identity functor $\id$ of $\tt Fun(R,M)$ is a left Quillen equivalence from the projective model structure to the Reedy model structure and from the Reedy structure to the injective one and a right Quillen equivalence in the other direction \cite{JL}:
\be\label{QuiEqu-s}\id: {\tt Fun}_{\op{proj}}{\tt (R,M)}\rightleftarrows{\tt Fun}_{\op{Reedy}}{\tt(R,M)}\rightleftarrows{\tt Fun}_{\op{inj}}{\tt (R,M)}:\id\;.\ee

\begin{cor}\label{QuiEquCor}
Any projective cofibration is a Reedy cofibration and any Reedy cofibration is an object-wise cofibration. Dually, any injective fibration is a Reedy fibration and any Reedy fibration is an object-wise fibration.
\end{cor}

\section{Indeterminacy of a homotopy pullback}

Let $\tt M$ now be any model category and let $\tt R$ be the inverse Reedy category ${\tt I}_2$ whose underlying category is ${\tt I}:=\{c\to d\leftarrow b\}$ and whose degree map is the above-mentioned map $\op{deg}_2$ defined by $\{1\to 0\leftarrow 2\}\,.$ The objects $X$ of the functor category $${\tt M}^{{\tt I}}:=\tt Fun(I,M)$$ are the $\tt M$-cospans $C\to D\leftarrow B$ and its morphisms $f:X\Rightarrow Y$ are the corresponding adjacent commutative squares

\begin{equation}
\begin{tikzpicture}
 \matrix (m) [matrix of math nodes, row sep=3em, column sep=2.5em]
   {C & & D & & B \\
    C' & & D' & & B' \\};
 \path[->]
 (m-1-1) edge [->] (m-1-3)
 (m-1-3) edge [<-] (m-1-5)
 (m-2-1) edge [->] (m-2-3)
 (m-2-3) edge [<-] (m-2-5)
 (m-1-1) edge [->] node[left] {\small{$f_c$}} (m-2-1)
 (m-1-3) edge [->] node[left] {\small{$f_d$}}(m-2-3)
 (m-1-5) edge [->] node[right] {\small{$f_b$}}(m-2-5);
\end{tikzpicture}
\end{equation}
In view of Remark \ref{RDI} the Reedy model structure on ${\tt M}^{{\tt I}_2}$ is the injective model structure of $\tt M^I$ with object-wise weak equivalences and cofibrations. Further, a natural transformation $f:X\Rightarrow Y$ is an injective fibration if and only if Condition \eqref{ReeFib} is satisfied. It follows from the definition of matching objects of objects $X\in{\tt M}^{{\tt I}_2}$ that $M_bX=D,\, M_cX=D$ and $M_dX=\ast\,,$ so that $f$ is an injective fibration if and only if the induced universal $\tt M$-morphisms are fibrations: $$\label{InjFibCos}B\twoheadrightarrow B'\times_{D'}D\,,\quad C\twoheadrightarrow C'\times_{D'}D\,,\quad D\twoheadrightarrow D'\;.$$ In particular:

\begin{prop}
For any model category $\tt M\,$, the injective model structure on the category of $\tt M$-cospans exists. Moreover, an $\tt M$-cospan $C\to D\leftarrow B$ is injectively fibrant if and only if $D$ is a fibrant object of $\,\tt M$ and both arrows are fibrations of $\,\tt M:$ \be\label{InjFib}C\twoheadrightarrow D_{\op{f}}\twoheadleftarrow B\;.\ee
\end{prop}

If ${\tt I}_3$ is the non-trivial Reedy category ${\tt I}=\{c\to d\leftarrow b\}$ with degree map $\op{deg}_3$ defined by $\{0\to 1\leftarrow 2\}\,,$ the computation of the Reedy fibrations is the same as in the case of ${\tt I}_2\,,$ except that $M_cX=\ast\,,$ so that $f$ is a fibration of the Reedy model structure $\op{Ree}_{\op{I}}$ defined by the increasing labelling $\{0\to 1\leftarrow 2\}$ if and only if $$\label{Ree3FibCos}B\twoheadrightarrow B'\times_{D'}D\,,\quad C\twoheadrightarrow C'\,,\quad D\twoheadrightarrow D'\;.$$ Dually, a natural transformation $f$ is a cofibration of the Reedy model structure $\op{Ree}_{\op{I}}$ if and only if \be\label{Ree3CoFibCos}B\rightarrowtail B'\,,\quad C\rightarrowtail C'\,,\quad D\amalg_CC'\rightarrowtail D'\;.\ee In particular:

\begin{prop}
For any model category $\tt M\,,$ an $\tt M$-cospan $C\to D\leftarrow B$ is fibrant for the Reedy model structure $\op{Ree}_{\op{I}}$ defined by the increasing labelling $\{0\to 1\leftarrow 2\}\,$ if and only if $C$ and $D$ are fibrant objects of $\,\tt M$ and the second arrow is a fibration of $\,\tt M:$ \be\label{ReeIFib}C_{\op{f}}\rightarrow D_{\op{f}}\twoheadleftarrow B\;.\ee
\end{prop}

If ${\tt M}^{{\tt I}}$ is equipped with its Reedy structure $\op{Ree}_{\op{I}}\,,$ the constant functor \be\label{AdjLimReeI}-^\ast:{\tt M}\rightleftarrows {\tt Fun}_{\op{Ree}_{\op{I}}}({\tt I},{\tt M}):\op{Lim}\ee preserves cofibrations. Indeed, the image by $-^\ast$ of an $\tt M$-morphism $m:E\to E'$ is the commutative diagram
\begin{equation}
\begin{tikzpicture}
 \matrix (m) [matrix of math nodes, row sep=3em, column sep=2.5em]
   {E & & E & & E \\
    E' & & E' & & E' \\};
 \path[->]
 (m-1-1) edge [->] node[above] {\small{$\id_E$}} (m-1-3)
 (m-1-3) edge [<-] node[above] {\small{$\id_E$}} (m-1-5)
 (m-2-1) edge [->] node[above] {\small{$\id_{E'}$}}(m-2-3)
 (m-2-3) edge [<-] node[above] {\small{$\id_{E'}$}} (m-2-5)
 (m-1-1) edge [->] node[left] {\small{$m$}} (m-2-1)
 (m-1-3) edge [->] node[left] {\small{$m$}}(m-2-3)
 (m-1-5) edge [->] node[right] {\small{$m$}}(m-2-5);
\end{tikzpicture}
\end{equation}
and if $m:E\rightarrowtail E'\,,$ this diagram is a cofibration of $\op{Ree}_{\op{I}}$ if and only if the conditions \eqref{Ree3CoFibCos} are satisfied, i.e., if and only if the universal morphism $u:E\amalg_EE'\to E'$ induced by $m$ is a cofibration in $\tt M\,.$ One easily sees that $u=\id_{E'}\,,$ so the previous diagram is indeed a cofibration of $\op{Ree}_{\op{I}}\,.$ As weak equivalences are defined object-wise in any Reedy structure, the constant functor preserves also trivial cofibrations and therefore the adjunction \eqref{AdjLimReeI} is a Quillen adjunction. The right adjoint functor of the resulting adjunction in homotopy $$\mathbb{L}_{\op{Ree}_{\op{I}}}(-^\ast):\tt Ho(M)\rightleftarrows Ho(Fun_{\op{Ree}_{\op{I}}}(I,M)):\mathbb{R}_{\op{Ree}_{\op{I}}}\!\op{Lim}\;$$ is the K or the S homotopy limit functor with respect to the Reedy model structure $\op{Ree}_{\op{I}}$ (Definition \ref{HoLimFunDef}).\medskip

Similarly, if ${\tt I}_4$ is the Reedy category ${\tt I}=\{c\to d\leftarrow b\}$ with degree map $\op{deg}_4$ defined by the decreasing labelling $\{2\to 1\leftarrow 0\}\,,$ we get the

\begin{prop}
For any model category $\tt M\,,$ an $\tt M$-cospan $C\to D\leftarrow B$ is fibrant for the Reedy model structure $\op{Ree}_{\op{D}}$ defined by the decreasing labelling $\{2\to 1\leftarrow 0\}$ if and only if $D$ and $B$ are fibrant objects of $\,\tt M$ and the first arrow is a fibration of $\,\tt M\,:$ \be\label{ReeDFib}C\twoheadrightarrow D_{\op{f}}\leftarrow B_{\op{f}}\;.\ee
\end{prop}

Moreover, just as in the case of $\op{Ree}_{\op{I}}\,,$ there is a K and an S homotopy limit functor $\mathbb{R}_{\op{Ree}_{\op{D}}}\!\op{Lim}$ with respect to $\op{Ree}_{\op{D}}\,$.

\begin{defi}
Let $\tt I$ be the category $\{c\to d\leftarrow b\}\,,$ let $\tt M$ be a model category and let $\zs$ be a model structure on the category $\tt Fun(I,M)$ of cospans of $\,\tt M$ such that the adjunction \eqref{AdjLim} is a Quillen adjunction $-^\ast\dashv\op{Lim}.$ In particular $\zs\in\{\op{inj},\op{Ree}_{\op{I}},\op{Ree}_{\op{D}}\}\,.$ For every $\,\tt M$-cospan $X=\{C\to D\leftarrow B\}\,,$ its homotopy limit with respect to $\zs$ \be\label{IndeterminacyHoLimSigma}\mathbb{R}_{\zs}\!\op{Lim}(X)\zig\op{Lim}(R_{\zs}X)\stackrel{\sim}{\rightleftarrows}\op{Lim}(\tilde{F}_{\zs}X)
\stackrel{\sim}{\to}\op{Lim}(F_{\zs}X)\ee (Theorem \ref{IndeterminacyHoLimTheo}) is referred to as the {\bf homotopy pullback} of $X$ with respect to $\zs$ and it is denoted $$B\times_D^{h_\zs} C:=\mathbb{R}_{\zs}\!\op{Lim}(C\to D\leftarrow B)\;.$$
\end{defi}

If $F_{\zs 1}X$ and $F_{\zs 2}X$ are two fibrant replacements of $X$ in $\zs\,,$ it follows from \eqref{IndeterminacyHoLimSigma} that there is a span of weak equivalences $$\op{Lim}(F_{\zs 1}X)\stackrel{\sim}{\leftarrow}\cdot \stackrel{\sim}{\to}\op{Lim}(F_{\zs 2}X)\;.$$ If $\zs=\op{Ree}_{\op{I}}\,,$ we get in particular \be\label{HoPullReeI}B\times_D^{h_{\op{Ree}_{\op{I}}}} C\approx \op{Lim}(G\twoheadrightarrow H_{\op{f}}\twoheadleftarrow E)\stackrel{\sim}{\leftarrow}\cdot \stackrel{\sim}{\to}\op{Lim}(L_{\op{f}}\to M_{\op{f}}\twoheadleftarrow K)\;,\ee for every cospans to which $X$ is weakly equivalent. Similarly, if $\zs=\op{Ree}_{\op{D}}\,,$ we obtain \be\label{HoPullReeD}B\times_D^{h_{\op{Ree}_{\op{D}}}} C\approx \op{Lim}(G\twoheadrightarrow H_{\op{f}}\twoheadleftarrow E)\stackrel{\sim}{\leftarrow}\cdot \stackrel{\sim}{\to}\op{Lim}(P\twoheadrightarrow S_{\op{f}}\leftarrow N_{\op{f}})\;,\ee whenever the cospans considered are replacements of $X\,$. Since the first replacement in the last two equations is also fibrant if $\zs=\op{inj}\,,$ we have \be\label{HoPullInj}B\times_D^{h_{\op{inj}}} C\approx \op{Lim}(G\twoheadrightarrow H_{\op{f}}\twoheadleftarrow E)\;.\ee

Of course, the standard limit of a cospan is its standard pullback.\medskip

What we said above leads to the next proposition which deals with all of the possible indeterminacies in homotopy pullbacks (see \eqref{HoPullReeI}, \eqref{HoPullReeD}, \eqref{HoPullInj}).

\begin{rem}
\emph{In every model category $\tt M$ the homotopy pullback of a cospan with respect to $\zs\in\{\op{inj},\op{Ree}_{\op{I}},\op{Ree}_{\op{D}}\}$ is well defined as an isomorphism class of objects of $\,\tt Ho(M)\,,$ but is only defined up to a zigzag of weak equivalences if it is viewed as an object of $\tt M\,.$ All types of fibrant replacement (fibrant C-replacement functor, local fibrant C-replacement, or just any fibrant replacement) provide representatives of the $\zs$-homotopy pullback considered, and this for both interpretations of the homotopy pullback (Kan extension derived functor or strongly homotopy derived functor). In addition, \emph{we can regard} the representatives of a homotopy pullback for the three model structures on cospans (injective model structure, Reedy model structure defined by the increasing labelling, or Reedy model structure defined by the decreasing labelling) as being the same. \emph{In this sense} the homotopy pullback is independent of the model structure on cospans.}
\end{rem}

\begin{theo}\label{IndHoPull}
The homotopy pullback of a cospan in a model category is independent of the type of derived functor and of the model structure $$\zs\in\{\op{inj},\op{Ree}_{\op{I}},\op{Ree}_{\op{D}}\}$$ on cospan diagrams considered. We get canonical representatives of the homotopy pullback from the standard pullback of weakly equivalent cospans with three fibrant objects and at least one morphism that is a fibration: more precisely, if in the adjacent commutative squares
\begin{equation}
\begin{tikzpicture}
 \matrix (m) [matrix of math nodes, row sep=3em, column sep=2.5em]
   {C & & D & & B \\
    C' & & D' & & B' \\};
 \path[->]
 (m-1-1) edge [->] (m-1-3)
 (m-1-3) edge [<-] (m-1-5)
 (m-2-1) edge [->] (m-2-3)
 (m-2-3) edge [<-] (m-2-5)
 (m-1-1) edge [->] node[left] {\small{$\sim$}} (m-2-1)
 (m-1-3) edge [->] node[left] {\small{$\sim$}}(m-2-3)
 (m-1-5) edge [->] node[right] {\small{$\sim$}}(m-2-5);
\end{tikzpicture}
\end{equation}
all vertical arrows are weak equivalences, all bottom nods are fibrant objects and at least one of the bottom arrows is a fibration, we have \be\label{HoPull}B\times_D^{h} C\approx B'\times_{D'}C'\;.\ee
\end{theo}

\begin{rem}
\emph{In other words, we consider the {\bf full homotopy pullback} (or simply the homotopy pullback) $B\times_D^{h} C\,,$ whose {\bf canonical representatives} are the standard pullbacks of weakly equivalent cospans whose three nods are fibrant and at least one of whose arrows is a fibration.}
\end{rem}

\section{Models of a homotopy pullback}

In this section we extend the concept of representative of a homotopy limit under the name of homotopy limit model.\medskip

Let $\,\tt S$ be a small category, let $\tt M$ be a model category and let $\zs$ be a model structure on the category $\tt Fun(S,M)$ such that the adjunction $$-^\ast:{{\tt M}\rightleftarrows {\tt Fun}_\zs({\tt S,M})}:\op{Lim}$$ is a Quillen adjunction $-^\ast\dashv\op{Lim}\,.$ If $X\in\tt Fun(S,M)$ and $F_\zs X$ is a fibrant replacement $$t_{F_\zs X}\circ f_{X}:X\stackrel{\sim}{\to}F_\zs X\twoheadrightarrow\ast$$ of $X\,,$ the universal morphism $$\op{Lim}(f_{X}):\op{Lim}X\to\op{Lim}(F_\zs X)$$ from the limit $\op{Lim}X$ of $X$ to the representative $\op{Lim}(F_\zs X)$ of the homotopy limit $\mathbb{R}_\zs\!\op{Lim}(X)$ of $X$ is usually not a weak equivalence.

\begin{defi}\label{HomLimMod}
Let $\tt S\,,$ $\tt M$ and $\zs$ be as above, and let $A\in\tt M\,,$ $X\in\tt Fun(S,M)$ and $$\za\in\op{Hom}_{\tt Fun(S,M)}(A^*,X)\cong\op{Hom}_{\tt M}(A,\op{Lim}X)\ni\op{Lim}\za\;.$$ We say that $A$ is a {\bf generalized representative} of the $\zs$-homotopy limit of $X$ or is a $\zs$-homotopy limit {\bf model} of $X,$ if there exists a fibrant replacement $F_\zs X$ of $X$ such that the composite of universal morphisms $$\op{Lim}(f_{X})\circ\op{Lim}\za: A\to\op{Lim}X\to\op{Lim}(F_\zs X)\;$$ is a weak equivalence.
\end{defi}

\begin{prop}\label{IndHoLimModProp}
If the condition in Definition \ref{HomLimMod} is satisfied for one fibrant replacement, it holds also for every other fibrant replacement.
\end{prop}

\begin{proof}
Let $F_\zs'X$ be another fibrant replacement of $X$ and let $\tilde{F}_\zs X$ be a fibrant C-replacement:

\begin{equation} \begin{tikzpicture}
 \matrix (m) [matrix of math nodes, row sep=1.5em, column sep=2em]
   {  &&\text{$A^*$}&&\\ &&&&\\&&\text{$X$}&& \\ &&&&\\ \text{$F_\zs X$}
       && \text{$\tilde{F}_\zs X$} && \text{$F_\zs' X$} \\ };
 \path[->]
 (m-1-3) edge node[auto] {\small{$\za$}} (m-3-3)
 (m-3-3) edge node[above] {\small{$f_X\;\;$}} (m-5-1)
 (m-3-3) edge node[auto] {\small{$\tilde{f}_X$}} (m-5-3)
 (m-3-3) edge node[auto] {\small{$f'_X$}} (m-5-5)
 (m-5-1) edge[<-] node[below] {\small{$\ell_\zs$}} (m-5-3)
 (m-5-3) edge node[below] {\small{$\ell'_\zs$}} (m-5-5);
\end{tikzpicture}
\end{equation}
Recall that the liftings $\ell_\zs$ and $\ell'_\zs$ in the previous commutative triangles are weak equivalences since $f_X,\tilde{f}_X$ and $f'_X$ are (see \eqref{FibRep}). As \be\label{IndHoLimMod}\op{Lim}(f_X)\circ\op{Lim}\za=\op{Lim}(f_X\circ\za)=\op{Lim}(\ell_\zs\circ\tilde{f}_X\circ\za)=\op{Lim}(\ell_\zs)\circ{\op{Lim}}(\tilde{f}_X\circ\za)\;,\ee it follows from \eqref{UniMorHoLim} that ${\op{Lim}}(\tilde{f}_X\circ\za)$ is a weak equivalence, and it follows from \eqref{IndHoLimMod} written for $f'_X$ and $\ell'_\zs$ and from \eqref{UniMorHoLim} that $\op{Lim}(f'_X)\circ\op{Lim}\za$ is a weak equivalence.
\end{proof}

In the special case of the homotopy pullback the category $\tt S$ is ${\tt I}=\{c\to d\leftarrow b\}$ and $X$ is a cospan $\{C\to D\leftarrow B\}\,.$ The natural transformation $\za$ is made of adjacent commutative squares whose top row $A\to A\leftarrow A$ contains two copies of $\id_A\,,$ or, better, is made of a single commutative square
\begin{equation}\label{ComSqu}
\begin{tikzpicture}
 \matrix (m) [matrix of math nodes, row sep=2.5em, column sep=1.5em]
   {A && B\\
    C && D\\};
 \path[->]
 (m-1-1) edge [->] (m-1-3)
 (m-2-1) edge [->] (m-2-3)
 (m-1-1) edge [->] (m-2-1)
 (m-1-3) edge [->] (m-2-3);
\end{tikzpicture}
\end{equation}
and $\op{Lim}\za$ is the universal morphism $A\to B\times_DC\,.$ The replacement $F_\zs X$ is a fibrant cospan $C'\to D'\leftarrow B'$ to which $C\to D\leftarrow B$ is weakly equivalent; its pullback $B'\times_{D'}C'$ is a representative of $B\times_D^{h_\zs}C\,.$ The composite $\op{Lim}(f_X)\circ\op{Lim}\za$ of universal morphisms is the universal morphism $A\to B\times_DC\to B'\times_{D'}C'$ from $A$ to the representative of $B\times_D^{h_\zs}C$ considered. Hence, the definition \eqref{HomLimMod} becomes:

\begin{defi}\label{HomPulMod}
Let $\tt M$ be a model category and let $\zs$ be a model structure on the category of cospans of $\,\tt M$ such that $-^\ast\dashv \op{Lim}$ is a Quillen adjunction. The vertex $A$ of a commutative square \eqref{ComSqu} is a {\bf model or generalized representative} of the $\zs$-homotopy pullback $B\times^{h_\zs}_DC$ if there exists a fibrant replacement $C'\to D'\leftarrow B'$ of $\,C\to D\leftarrow B$ in $\zs$ such that the universal morphism $A\to B'\times_{D'}C'$ from $A$ to the representative of $B\times^{h_\zs}_DC$ considered, is a weak equivalence.
\end{defi}

The condition in Definition \ref{HomPulMod} is satisfied for every fibrant replacement in $\zs$ if it is satisfied for one of them. As mentioned in the proof of Proposition \ref{IndHoLimModProp}, this independence of the replacement is due to \eqref{UniMorHoLim}, therefore it is a consequence of the fact that the limit functor preserves weak equivalences between fibrant objects; so it is ultimately a consequence of the assumption that $-^\ast\dashv \op{Lim}$ is a Quillen adjunction.\medskip

If we confine ourselves to the model structures $\zs\in\{\op{inj},\op{Ree}_{\op{I}},\op{Ree}_{\op{D}}\}\,,$ the definition is not only independent of the replacement, but also of the model structure in which this replacement is chosen:

\begin{theo}\label{IndHoPullMod}
The vertex $A$ of a commutative square \eqref{ComSqu} in a model category is a model or \emph{generalized representative} of the full homotopy pullback $B\times_D^{h}C$ if the universal morphism from $A$ to a \emph{canonical representative} of $B\times_D^{h}C$ is a weak equivalence. In other words, there must exist a cospan $C'\to D'\leftarrow B'$ to which $C\to D\leftarrow B$ is weakly equivalent, whose three nodes are fibrant objects and at least one of whose morphisms is a fibration, such that the universal morphism $A\to B'\times_{D'}C'$ is a weak equivalence.
\end{theo}

\begin{proof}
If the condition is satisfied for a fibrant replacement in one of the three model structures, it is satisfied for all the fibrant replacements in this model structure and in particular for the replacements of the type $G\twoheadrightarrow H_{\op{f}}\twoheadleftarrow E\,$. Hence, it is also satisfied for all the fibrant replacements in any of the other two model structures (see \eqref{HoPullReeI},\eqref{HoPullReeD},\eqref{HoPullInj}).\end{proof}

\begin{rem}\label{IndRep31}
\emph{We just showed that if the condition in Theorem \ref{IndHoPullMod} is satisfied for one replacement with three fibrant nodes and at least one fibration, it is satisfied for all replacements of this type. In other words, the concept of model is compatible with our identification of the homotopy pullbacks with respect to $\zs\in\{\op{inj}, \op{Ree}_{\op{I}},\op{Ree}_{\op{D}}\}\,.$}
\end{rem}

If the model category is right proper, we can weaken the model condition:

\begin{theo}\label{RightProper}
The vertex $A$ of a commutative square \eqref{ComSqu} in a \emph{right proper} model category is a model of the full homotopy pullback $B\times_D^{h}C$ if there exists a cospan $C'\to D'\leftarrow B'$ to which $C\to D\leftarrow B$ is weakly equivalent and at least one of whose morphisms is a fibration, such that the universal morphism $A\to B'\times_{D'}C'$ is a weak equivalence.
\end{theo}

\begin{lem}\label{PasLem} Let $\tt M$ be a right proper model category and let $f:A\to D\,,$ $g:B\to C$ and $h:C\to D$ be morphisms in $\tt M\,$. The pullbacks $A\times_DB$ and $A\times_DC$ exist and there is a universal morphism $u:A\times_DB\to A\times_DC\,.$ If $f:A\twoheadrightarrow D$ and $g:B\stackrel{\sim}{\to} C\,,$ we have $u:A\times_DB\stackrel{\sim}{\dashrightarrow} A\times_DC\,:$
\begin{equation}
\begin{tikzpicture}
 \matrix (m) [matrix of math nodes, row sep=3em, column sep=2.5em]
   {\text{$A\times_DB$} & & \text{$A\times_DC$} & & A \\
    B & & C & & D \\};
 \path[dashed,->] (m-1-1)  edge[pil] node[above] {\small{$\;\sim\;\;\;u$}} (m-1-3);
 \path[->]
 (m-1-3) edge [->] (m-1-5)
 (m-2-1) edge [->] node[above] {\small{$\;\sim\;\;\; g$}} (m-2-3)
 (m-2-3) edge [->] node[auto] {\small{$h$}} (m-2-5)
 (m-1-1) edge [->] (m-2-1)
 (m-1-3) edge [->] node[auto] {\small{$k$}} (m-2-3)
 (m-1-5) edge [->>] node[right] {\small{$f$}} (m-2-5);
\end{tikzpicture}
\end{equation}
\end{lem}
\begin{proof}
Recall the {\bf pasting law for pullbacks}: if in adjacent commutative squares in a category the right square is a pullback, then the left square is a pullback if and only if the total square is a pullback. In our case it follows that the left square is a pullback: $A\times_DB\cong(A\times_DC)\times_CB\,.$ Since fibrations are closed under pullbacks, we get that $k$ is a fibration. As a right proper model category is a model category in which pullbacks along fibrations preserve weak equivalences, we see that $u$ is a weak equivalence.
\end{proof}

\begin{proof}[Proof of Theorem \ref{RightProper}]
Assume that $C'\stackrel{g}{\to} D'\stackrel{f}{\leftarrow} B'$ is a replacement of $C\to D\leftarrow B$ and that one of its morphisms is a fibration, for instance the {\it second} one. If we apply a fibrant C-replacement functor $R$ to $C'\stackrel{g}{\to} D'\stackrel{f}{\twoheadleftarrow} B'$ and decompose the {\it first} arrow $RC'\stackrel{Rg}{\to} RD'$ into $RC'\stackrel{\sim}{\to}F(Rg)\twoheadrightarrow RD'\,,$ we get the commutative diagram
\begin{equation}\label{}
\begin{tikzpicture}
 \matrix (m) [matrix of math nodes, row sep=1.5em, column sep=1.5em]
   {&& A && \\
   C && D && B\\
   C' && D' && B'\\
   RC' && RD' && RB'\\
   F(Rg) && RD' && RB'\\};
 \path[->]
 (m-1-3) edge [->] (m-2-1)
 (m-1-3) edge [->] (m-2-5)
 (m-2-1) edge [->] (m-2-3)
 (m-2-3) edge [<-] (m-2-5)
 (m-2-1) edge [->] node[auto] {\small{$\sim$}}(m-3-1)
 (m-2-3) edge [->] node[auto] {\small{$\sim$}}(m-3-3)
 (m-2-5) edge [->] node[auto] {\small{$\sim$}}(m-3-5)
 (m-3-1) edge [->] node[auto] {\small{$g$}} (m-3-3)
 (m-3-3) edge [<<-] node[auto] {\small{$f$}} (m-3-5)
 (m-3-1) edge [>->] node[auto] {\small{$\sim$}}(m-4-1)
 (m-3-3) edge [>->] node[auto] {\small{$\sim$}}(m-4-3)
 (m-3-5) edge [>->] node[auto] {\small{$\sim$}}(m-4-5)
 (m-4-1) edge [->] node[auto] {\small{$Rg$}} (m-4-3)
 (m-4-3) edge [<-] node[auto] {\small{$Rf$}} (m-4-5)
 (m-4-1) edge [->] node[auto] {\small{$\sim$}} (m-5-1)
 (m-4-3) edge [->] node[auto] {\small{$\id$}} (m-5-3)
 (m-4-5) edge [->] node[auto] {\small{$\id$}} (m-5-5)
 (m-5-1) edge [->>] (m-5-3)
 (m-5-3) edge [<-] node[auto] {\small{$Rf$}} (m-5-5) ;
\end{tikzpicture}
\end{equation}
From the 2-out-of-3 axiom it follows that there is a universal weak equivalence $$C'\stackrel{\sim}{\dashrightarrow} D'\times_{RD'}F(Rg)\;,$$ as the model category is right proper. In view  of Lemma \ref{PasLem}, we have now universal weak equivalences $$B'\times_{D'}C'\stackrel{\sim}{\dashrightarrow} B'\times_{D'}(D'\times_{RD'}F(Rg))\cong B'\times_{RD'}F(Rg)\stackrel{\sim}{\dashrightarrow}RB'\times_{RD'}F(Rg)\;.$$ Hence, the universal morphism \be\label{Former}A\dashrightarrow B'\times_{D'}C'\ee is a weak equivalence if and only if the universal morphism \be\label{Former2}A\dashrightarrow RB'\times_{RD'}F(Rg)\ee is a weak equivalence. Since the cospan $F(Rg)\twoheadrightarrow RD'\leftarrow RB'$ is weakly equivalent to $C\to D\leftarrow B\,,$ has three fibrant nodes and at least one of its morphisms is a fibration, the vertex $A$ of the square \eqref{ComSqu} is a model of $B\times_D^hC\,,$ if \eqref{Former} is a weak equivalence.
\end{proof}

\begin{prop}\label{IndRep1}
If the condition in Theorem \ref{RightProper} is satisfied for one replacement with one fibration it is satisfied for all replacements of this type.
\end{prop}

\begin{proof}
We see from Equations \eqref{Former} and \eqref{Former2} that the condition is satisfied for a given replacement with one fibration if and only if it is satisfied for an associated replacement with three fibrant nodes and one fibration. However, from Remark \ref{IndRep31} we know that if the condition is satisfied for one replacement of the latter type it is satisfied for all replacements of this type.\end{proof}

The following corollary is stated without proof in \cite{JL}:

\begin{cor}\label{CorLur}
In a model category the standard pullback $B\times_DC$ of a cospan $C\stackrel{g}{\to} D\stackrel{f}{\leftarrow} B$ is a model of the cospan's homotopy pullback if at least one of the morphisms $f$ or $g$ is a fibration and either all three objects $B,C,D$ are fibrant or the model category is right proper.
\end{cor}

\begin{proof}
Under the stated conditions $B\times_DC$ is a model of $B\times_D^hC\,.$ Indeed, if the model category is right proper (resp., $B,C$ and $D$ are fibrant), the cospan $C\to D\leftarrow B$ is a replacement of itself, one of its morphisms is a fibration (resp., and all its nodes are fibrant), and the universal morphism $\id:B\times_DC\dashrightarrow B\times_DC$ is a weak equivalence. The conclusion now follows from Theorem \ref{RightProper} (resp., Theorem \ref{IndHoPullMod}).
\end{proof}

Further, the concept of model of a homotopy pullback captures the notion of homotopy fiber square defined in \cite{Hir} and puts it in the right context.\medskip

Let $(\za,\zb)$ be a {\bf f}ixed {\bf f}unctorial trivial cofibration - fibration {\bf f}actorization system (FFF for short) of a model category and let $C\stackrel{g}{\to} D\stackrel{f}{\leftarrow}B$ be a cospan. The system considered provides decompositions $$C\stackrel{\sim}{\rightarrowtail}\zX(g)\twoheadrightarrow D\twoheadleftarrow\zX(f)\stackrel{\sim}{\leftarrowtail}B\;.$$

\begin{defi}\label{HoFibSqDef}
Let $\,\tt M$ be a right proper model category with an FFF. A commutative square
\begin{equation}\label{ComSqu2}
\begin{tikzpicture}
 \matrix (m) [matrix of math nodes, row sep=2.5em, column sep=1.5em]
   {A && B\\
    C && D\\};
 \path[->]
 (m-1-1) edge [->] (m-1-3)
 (m-2-1) edge [->] node[auto] {\small{$g$}} (m-2-3)
 (m-1-1) edge [->] (m-2-1)
 (m-1-3) edge [->] node[auto] {\small{$f$}} (m-2-3);
\end{tikzpicture}
\end{equation}
is a {\bf homotopy fiber square} if the universal morphism $A\dashrightarrow \zX(f)\times_D\zX(g)$ is a weak equivalence.
\end{defi}

\begin{cor}\label{HoFibSqCor}
In a right proper model category with an FFF, the vertex $A$ of a commutative square \eqref{ComSqu2} is a model of the homotopy pullback $B\times_D^hC$ if and only if it is a homotopy fiber square.
\end{cor}

\begin{proof} In view of Proposition \ref{IndRep1}, since the second row in the commutative diagram
\begin{equation}\label{}
\begin{tikzpicture}
 \matrix (m) [matrix of math nodes, row sep=1.5em, column sep=1.5em]
   {&& A && \\
   C && D && B\\
   \zX(g) && D && \zX(f)\\};
 \path[->]
 (m-1-3) edge [->] (m-2-1)
 (m-1-3) edge [->] (m-2-5)
 (m-2-1) edge [->] node[auto]  {\small{$g$}} (m-2-3)
 (m-2-3) edge [<-] node[auto] {\small{$f$}} (m-2-5)
 (m-2-1) edge [>->] node[auto] {\small{$\sim$}}(m-3-1)
 (m-2-3) edge [->] node[auto] {\small{$\id$}}(m-3-3)
 (m-2-5) edge [>->] node[auto] {\small{$\sim$}}(m-3-5)
 (m-3-1) edge [->>] (m-3-3)
 (m-3-3) edge [<<-] (m-3-5);
 \end{tikzpicture}
\end{equation}
is a replacement of the first row with at least one fibration, the vertex $A$ of the commutative triangle or square is a model of the homotopy pullback $B\times_D^hC$ if and only if the universal morphism $A\dashrightarrow \zX(f)\times_D\zX(g)$ is a weak equivalence, i.e., if and only if the square is a homotopy fiber square.
\end{proof}

\begin{rem}
\emph{Our philosophy has been to refer to the upper left vertex of a commutative square as a model for the homotopy pullback of the square's cospan when the universal morphism from it to a canonical representative of the homotopy pullback is a weak equivalence. In view of Corollary \ref{HoFibSqCor} and Definition \ref{HoFibSqDef} it makes therefore sense to regard the standard pullback $\zX(f)\times_D\zX(g)$ as a representative of $B\times_D^hC\,,$ provided the underlying model category is right proper and equipped with an FFF. Actually the homotopy pullback $B\times_D^hC$ is defined in \cite{Hir} as being this representative. If the lower right vertex of the square is fibrant, the homotopy pullback of \cite{Hir}, which is well defined as an object of the model category, is a canonical representative of our full homotopy pullback, which is only defined up to a zigzag of weak equivalences.}
\end{rem}

Next we prove that there is a {\bf pasting law for model squares}.

\begin{prop}\label{PasLaw}
Let
\begin{equation}\label{ComDia}
\begin{tikzpicture}
 \matrix (m) [matrix of math nodes, row sep=3em, column sep=2.5em]
   {A & & B & & C \\
    D & & E & & F \\};
 \path[->]
 (m-1-1) edge [->] (m-1-3)
 (m-1-3) edge [->] (m-1-5)
 (m-2-1) edge [->] (m-2-3)
 (m-2-3) edge [->] (m-2-5)
 (m-1-1) edge [->] (m-2-1)
 (m-1-3) edge [->] (m-2-3)
 (m-1-5) edge [->] (m-2-5);
\end{tikzpicture}
\end{equation}
be a commutative diagram in a model category. If the right square is a model square, i.e., if $B$ is a model of the homotopy pullback $C\times_F^h E\,,$ then the left square is a model square if and only if the total square is a model square.
\end{prop}

\begin{proof}
We apply a fibrant C-replacement functor $R$ to the commutative diagram \eqref{ComDia} and factor the morphism $$R(C\stackrel{\zk}{\rightarrow}F)=RC\stackrel{R\zk}{\longrightarrow}RF=RC\stackrel{\sim}{\rightarrow}F(R\zk){\twoheadrightarrow} RF$$ into a weak equivalence followed by a fibration. Moreover, we set $P:=F(R\zk)\times_{RF}RE$ and $Q:=P\times_{RE} RD$ and thus get the following commutative diagram:

\vspace{0.25cm}

\begin{equation}\label{AdjComCub}
\begin{tikzcd}[back line/.style={densely dotted}, row sep=1em, column sep=1em]
A \ar[dr, "\sim"] \ar{rrr} \ar{ddd} \ar[ddrr, bend right, dashed]
  & & & B \ar{ddd} \ar[dr, "\sim"] \ar{rrr} \ar[ddrr, bend right, dashed] & & & C \ar[dr, "\sim"] \ar{ddd} & & \\
& RA \ar{rrr} \ar{ddd} \ar[dr, dashed]
  & & & RB \ar{rrr} \ar[dr, dashed] & & & RC \ar{ddd} \ar[dr, "\sim"] & \\
& & Q \ar{rrr} \ar{ddd} & & & P \ar{rrr} & & & F(R\zk) \ar[ddd, twoheadrightarrow] \\
D \ar{rrr} \ar[dr, "\sim"] & & & E \ar[dr, "\sim"] \ar{rrr} & & & F \ar[dr, "\sim"] & &\\
& RD \ar{rrr} \ar[dr, equal] & & & RE \ar[dr, equal] \ar{rrr} \ar[crossing over, leftarrow]{uuu} & & & RF \ar[dr, equal] &\\
& & RD \ar{rrr} & & & RE \ar{rrr} \ar[crossing over, leftarrow]{uuu} & & & RF
\end{tikzcd}
\end{equation}

\vspace{0.65cm}

\noindent As the universal arrow $B\dashrightarrow P$ (resp., $A\dashrightarrow Q$) is the unique arrow $B\to P$ (resp., $A\to Q$) that makes the corresponding triangles commutative, this arrow coincides with the composite $B\stackrel{\sim}{\to}RB\dashrightarrow P$ (resp., $A\stackrel{\sim}{\to}RA\dashrightarrow Q$). Since the right square of \eqref{ComDia} is a model square, the universal arrow $B\dashrightarrow P$ is a weak equivalence in view of Remark \ref{IndRep31}, and therefore the universal arrow $RB\dashrightarrow P$ is a weak equivalence. In view of closeness of fibrations under pullbacks the arrow $P\to RE$ is a fibration. From here it follows that the left square in \eqref{ComDia} is a model square if and only if $$A\dashrightarrow Q=P\times_{RE}RD\cong F(R\zk)\times_{RF}RD$$ is a weak equivalence, which is the case if and only if the total square of \eqref{ComDia} is a model square.
\end{proof}

\vspace{0.25cm}

The next result is valid for homotopy fiber squares \cite{Hir} in a right proper model category with an FFF. We prove that it hold also for model squares in an arbitrary model category.\medskip

\begin{prop}\label{ComPara}
Let $ABCD$ and $A'B'C'D'$ be two commutative squares in a model category $\tt M\,$. If there exist four $\tt M$-morphisms from the vertices of the first square to the corresponding vertex of the second such that the four resulting squares commute and if these $\tt M$-morphisms are weak equivalences, then the first square is a model square if and only if the second is.
\end{prop}

\begin{proof}
First we apply a fibrant C-replacement functor $R$ to the commutative parallelepiped, which is described in the statement of Proposition \ref{ComPara}. Then we factor the morphism $$R(B\stackrel{\zk}{\rightarrow}D)=RB\stackrel{R\zk}{\longrightarrow}RD=RB\stackrel{\sim}{\rightarrow}F(R\zk){\twoheadrightarrow} RD$$ into a weak equivalence followed by a fibration and proceed analogously for $R(B'\stackrel{\zk'}{\to}D')\,,$ using a {\it functorial factorization system}. We also set $P:=F(R\zk)\times_{RD}RC$ and $P':=F(R\zk')\times_{RD'}RC'$. Since the factorization system used is functorial, we get an arrow $F(R\zk)\stackrel{\sim}{\rightsquigarrow} F(R\zk')\,,$ and thus a universal arrow $P\dashrightarrow P'\,.$ Finally, we have the commutative diagram \eqref{BigDia}.\medskip

The two commutative parallelograms with four red vertices in \eqref{BigDia} are a weak equivalence from the $\op{Ree}_{\op{I}}$-fibrant cospan $RC\to RD\twoheadleftarrow F(R\zk)$ to the $\op{Ree}_{\op{I}}$-fibrant cospan $RC'\to RD'\twoheadleftarrow F(R\zk')\,.$ Since the limit or pullback functor transforms weak equivalences between $\op{Ree}_{\op{I}}$-fibrant cospans into weak equivalences, the universal arrow $P\dashrightarrow P'$ is a weak equivalence. As the square $ABCD$ (resp., $A'B'C'D'$) is a model square if and only if the universal arrow $A\dashrightarrow P$ (resp., $A'\dashrightarrow P'$) is a weak equivalence, it follows that $ABCD$ is a model square if and only if $A'B'C'D'$ is a model square.

\begin{equation}\label{BigDia}
\begin{tikzcd}[back line/.style={densely dotted}, row sep=1.35em, column sep=0.75em]
& & & & & & & & A' \ar[rrdd, bend right, dashed, color=magenta] \ar[rd, "\sim", thick]\ar[rrr, thick]\ar[ddd, thick] & & & B' \ar[rd, "\sim", thick]\ar[ddd, thick] & &  \\
& & & & & & & & & RA'\ar[rd, thick, dashed] \ar[rrr, thick, color=blue] \ar[ddd, thick, color=blue] & & & RB'\ar[rd, thick, "\sim"] \ar[ddd, thick, color=blue] &  \\
& & & & & & & & & & P' \ar[rrr, thick, color=red] \ar[ddd, thick, color=red] & & & \color{red}{F(R\zk')} \ar[ddd, thick, twoheadrightarrow, color=red] \\
& & & & & & & & C'\ar[rd, "\sim", thick] \ar[rrr, thick] & & & D'\ar[rd, "\sim", thick] & &  \\
& & & & & & & & & RC'\ar[rd, equal] \ar[rrr, thick, color=blue] & & & RD'\ar[rd, equal] &  \\
& & & & & & & & & & \color{red}{RC'} \ar[rrr, thick, color=red] & & & \color{red}{RD'} \\
A \ar[rrdd, bend right, dashed, color=magenta]\ar[rd, "\sim", thick]\ar[rrr, thick]\ar[ddd, thick]\ar[rrrrrrrruuuuuu, "\sim"] & & & B \ar[rd, "\sim", thick]\ar[ddd, thick]\ar[rrrrrrrruuuuuu, "\sim"] & & & & & & & & & & \\
& RA \ar[dr, thick, dashed]\ar[color=blue, rrrrrrrruuuuuu, "\sim"]\ar[rrr, thick, color=blue]\ar[ddd, thick, color=blue] & & & RB \ar[rd, thick, "\sim"]\ar[color=blue, rrrrrrrruuuuuu, "\sim"]\ar[ddd, thick, color=blue] & & & & & & & & \\
& & P \ar[color=red, rrrrrrrruuuuuu, dashed]\ar[rrr, thick, color=red]\ar[ddd, thick, color=red] & & & \color{red}{F(R\zk)} \ar[color=red, rrrrrrrruuuuuu, "\sim", rightsquigarrow]  \ar[, color=red, ddd, thick, twoheadrightarrow]& & & & & & \\
C\ar[rd, "\sim", thick]\ar[rrrrrrrruuuuuu, "\sim"] \ar[rrr, thick] & & & D\ar[rd, "\sim", thick]\ar[rrrrrrrruuuuuu, "\sim"] & & & & & & & & & &\\
& RC \ar[dr, thick, equal]\ar[color=blue, rrrrrrrruuuuuu, "\sim"]\ar[rrr, thick, color=blue] & & & RD \ar[dr, thick, equal]\ar[color=blue,rrrrrrrruuuuuu, "\sim"] & & & & & & & & \\
& & \color{red}{RC} \ar[rrr, thick, color=red]\ar[color=red, rrrrrrrruuuuuu, "\sim"]  & & & \color{red}{RD}\ar[color=red, rrrrrrrruuuuuu, "\sim"] & & & & & & \\
\end{tikzcd}
\end{equation}
\end{proof}

\section{Concluding remarks}

Building on ideas from works by Beilinson, Costello, Drinfeld, Gwilliam, Schreiber, Paugam, To\"en, Vezzosi, and Vinogradov \cite{BD04, CG1, P11, Paugam1, TV05, TV08, Vino}, Di Brino and two of the authors of the present paper have introduced derived algebraic geometry over the ring $\cD$ of differential operators of an underlying affine scheme, as a suitable framework for investigating the solution space of a system of partial differential equations up to symmetries \cite{KTRCR, HAC, PP}. The implementation of the associated research program requires that the tuple $$({\tt DG\cD M, DG\cD M, DG\cD A}, \tau, \mathbf{P})$$ be a homotopic algebraic geometric context (HAGC) in the sense of \cite {TV08}, where $\tt DG\cD M $ is the symmetric monoidal model category of differential graded $\cD$-modules, the subcategory $\tt DG\cD A$ is the model category of differential graded $\cD$-algebras, $\zt$ is an appropriate model pre-topology on the opposite category of $\tt DG\cD A$ and $\mathbf{P}$ is a compatible class of morphisms. We expect that the challenging proof of the HAGC theorem will be based on a generalization of the concept of homotopy fiber sequence and a generalization of Puppe's long exact sequence. We are convinced that the notion of model square or homotopy fiber square in an arbitrary model category stands now on a solid mathematical foundation and that this enables us to prove the HAGC theorem and thus to take an important step towards fully working through the above-mentioned program.

{}
\vfill
{\emph{email:} {\sf alisa.govzmann@uni.lu}; \emph{email:} {\sf damjan.pistalo@uni.lu}; \emph{email:} {\sf norbert.poncin@uni.lu}.}

\end{document}